\documentclass[12pt,reqno]{amsart}
\usepackage{cite}
\usepackage{amssymb,latexsym,amsmath,amsfonts}
\usepackage{latexsym}
\usepackage[mathscr]{eucal}
= 8.9in
\theoremstyle{plain}
\newtheorem{theorem}{\textbf{Theorem}}[section]
\newtheorem{corollary}[theorem]{Corollary}
\newtheorem{proposition}[theorem]{Proposition}
\newtheorem{remark}[theorem]{\textbf{Remark}}

\newtheorem{definition}[theorem]{Definition}

\newtheorem{example}[theorem]{\textbf{Example}}

\numberwithin{equation}{section}
\begin{document}

\title[A new approach of couple Fixed Point results on JS-metric Spaces ]
{A new approach of couple Fixed Point results on JS-metric Spaces}
\author[ T. Senapati \& L.K. Dey]%
{Tanusri Senapati$^{1}$ and  Lakshmi Kanta Dey$^{2}$}

%\thanks{}

\address{{$^{1}$\,} Tanusri Senapati,
                    Department of Mathematics,
                    National Institute of Technology Durgapur,
                    West Bengal,
                    India.}
                    \email{senapati.tanusri@gmail.com}

\address{{$^{2}$\,} Lakshmi Kanta Dey,
                    Department of Mathematics,
                    National Institute of Technology Durgapur,
                    West Bengal,
                    India.}
                    \email{lakshmikdey@yahoo.co.in}

%\address{{$^{2}$\,} Stojan Radenovi\'{c},
%                    Faculty of Mechanical Engineering,
%                    University of Belgrade,
%                    Beograd,Serbia.}
%                    \email{radens@beotel.rs}

%\thanks{}

 \subjclass[2010]{ $47$H$10$, $54$E$50$. }
\begin{abstract}
In this article, we study coupled fixed point theorems in newly appeared JS-metric spaces. It is important to note that  the class of JS-metric spaces includes standard metric space, dislocated metric space, b-metric space etc. The purpose of this paper is to present several coupled fixed point results in a more general way. Moreover, the techniques used in our proofs are indeed different from the comparable existing literature. Finally, we present a non-trivial example to validate our main result.
\end{abstract}

 \maketitle
 \section{\textbf{Introduction and preliminaries}}
 Throughout this article, we use usual arithmetic operations in the set of (affinely) extended real number system $\mathbb{\bar{R}}=\mathbb{R}\cup \{+\infty,-\infty\}$ and the notations have their usual meaning. Let $X$ be a nonempty set and $D:X^2\rightarrow [0,\infty]$ be a mapping. For every $x\in X$, we consider the set $C(D,X,x)$ (see, \cite{jleli}) as follows:
\begin{equation}
C(D,X,x)=\{(x_n)\subset X: \lim_{n\rightarrow \infty}D(x_n,x)=0 \}.\nonumber
\end{equation}
Very recently, Jleli and Samet \cite{jleli} introduced an interesting generalization of metric space in the following way. 
\begin{definition} \cite{jleli}\label{dfn1}
Let $X$ be a nonempty set and $D:X^2\rightarrow [0,\infty]$ be a mapping. Then $(X,D)$ is said to be generalized metric space if the following conditions are satisfied:
\begin{enumerate}
\item[(D1)] $\forall x,y \in X, D(x,y)=0 \Rightarrow x=y$;
\item[(D2)] $\forall x,y \in X, D(x,y)=D(y,x)$;
\item[(D3)] there exists $c>0$ such that for all $(x,y) \in X^2$ and $(x_n)\in C(D,X,x)$
\begin{equation}
D(x,y)\leq c\limsup_{n\rightarrow \infty} D(x_n,y).\nonumber
\end{equation}
\end{enumerate}
If $C(D,X,x)=\phi$, then $(X,D)$ is generalized metric space if $D$ satisfies $(D1-D2)$. 
\end{definition} 
Throughout  this article, we call this metric space as a `\textbf{\textit{JS-metric space}}' (due to Jleli and Samet). The authors of \cite{jleli} reported that different abstract spaces such as standard metric space, dislocated metric space, b-metric space, modular space etc.  can be derived from their newly introduced metric space. They also established several fixed point results for the mappings satisfying famous Banach's contraction, $\acute{C}$iri$\acute{c}$ quasi-contraction, Banach's contraction in partially ordered metric spaces etc.  Inspirited by their work, we studied and established some more important results on this structure (see, \cite{sena2}).  For the notion of convergence, Cauchy sequence, completeness and other topological details, the readers are refereed  to see \cite{jleli} and \cite{sena2}.

In another direction, Bhaskar and Lakshmikantham \cite{bhas} introduced the concept of coupled fixed point in the setting of partially ordered metric space as follow:
\begin{definition}\cite{bhas}
An element $(x,y)\in X^2$ is said to be coupled fixed point of $F: X^2\rightarrow X$ if $x=F(x,y)$ and $y=F(y,x)$.
\end{definition}
They also introduced the concept of a mixed monotone operator given by:
\begin{definition}\cite{bhas}
Let $(X,\leq)$ be a partially ordered set and $F:X^2\rightarrow X$ be a function. Then $F$ is said to have mixed monotone property if $F$ has the following property: 
$$x_1\leq x_2\Rightarrow F(x_1,y)\leq F(x_2,y);\forall x_1,x_2,y\in X $$
and $$y_1\leq y_2\Rightarrow F(x,y_1)\geq F(x,y_2);\forall x,y_1,y_2\in X. $$
\end{definition}
Using this concept, the authors of \cite{bhas} presented the following result in support of the existence of coupled fixed point of an operator satisfying mixed monotone property in partially ordered complete metric spaces. 
\begin{theorem}\cite{bhas}
Let $(X,\leq)$ be a partially ordered set and $(X,d)$ be complete partially ordered metric space. Suppose $F:X^2\rightarrow X$ is a mixed monotone operator having the following properties:
\begin{equation}
d(F(x,y),F(u,v))\leq \frac{k}{2}\{d(x,u)+d(y,v)\} ~\forall x\geq u;y\leq v.\label{c1}
\end{equation}
 Also consider that there exist $x_0,y_0\in X$ with $x_0\leq F(x_0,y_0);y_0\geq F(y_0,x_0)$. If
\begin{enumerate}
\item[(A)]
$F$ is continuous or
\item[(B)]$X$ has the following property:
\begin{enumerate}
\item If a non-decreasing sequence $(x_n)\rightarrow x$ then $x_n\leq x$ for all $n\in \mathbb{N}$;
\item If a non-increasing sequence $(y_n)\rightarrow y$ then $y_n\geq y$ for all $n\in \mathbb{N}$
\end{enumerate}
\end{enumerate}
then there exist $x,y \in X$ such that $x=F(x,y)$ and $y=F(y,x)$.
\end{theorem}
Afterward, in 2011, Berinde \cite{berd} generalized the contraction condition \ref{c1} as follows:
\begin{equation}
d(F(x,y),F(u,v))+d(F(y,x),F(v,u))\leq k[d(x,u)+d(y,v)] \label{c2}
\end{equation}for all $x\geq u;y\leq v$ and established coupled fixed point for a mixed monotone operator in partially ordered complete metric spaces. For more results of fixed point and coupled fixed point, the readers may see \cite{guo,laks,stoj1,stoj2,isik1,isik2,sena1,lahiri}.

In this article, inspired by the ideas of JS-metric spaces, at first, we extend the coupled fixed point results of Bhaskar and Lakshmikantham \cite{bhas} due to contraction condition \ref{c1} for a mapping satisfying mixed monotone property in complete JS-metric spaces endowed with partial ordering. After that, we extend the coupled fixed point results due to Berinde \cite{berd} for the contraction condition \ref{c2}. It is notable that the triangular inequality, so called basic property of standard metric space is replaced by a more weaker condition in JS-metric spaces. Necessarily, the techniques used in our proofs are quite different  and most remarkably some of the proofs become  simpler. Finally  we construct a non-trivial example to substantiate  our main result.

 \section{\textbf{Main results}}
 In order to state our main results, we need to define some basic things regarding this structure. %\section{\textbf{Induced $JS$- metrics in higher dimensions}}
Let $(X,D)$ be a JS-metric space. Now we consider $X^2$ and define
$$D_+((x,y),(u,v))=D(x,u)+D(y,v)$$ for all $(x,y),(u,v)\in X^2.$ We prove that $(X^2, D_+)$
is a $D_+$-JS-metric space induced by the metric $D$.
\begin{enumerate}
\item[(i)] Let $D_+((x,y),(u,v))=0$. It implies that $D(x,u)+D(y,v)=0$. It is possible only when both $D(x,u)=0$ and $D(y,v)=0$ i.e. $x=u$ and $y=v$. Therefore, $$D_+((x,y),(u,v))=0\Rightarrow (x,y)=(u,v)$$ for all $(x,y),(u,v)\in X^2$.
\item[(ii)] Clearly, $D_+((x,y),(u,v))=D_+((u,v),(x,y))$ for all $(x,y),(u,v)\in X^2$.
\item[(iii)] Let $(x_n, y_n)\rightarrow (x,y)$ as $n\rightarrow \infty$. Then
\begin{eqnarray}
D_+((x,y),(u,v))&=&D(x,u)+D(y,v)\nonumber\\
&\leq &\limsup \{c_1D(x_n,u)+c_2D(y_n,v)\}\nonumber\\
&\leq & c_0 \limsup D_+((x_n,y_n),(u,v))\nonumber
\end{eqnarray}
where $c_0=\max\{c_1,c_2\}$.
\end{enumerate}
Thus $D_+$ satisfied all the axioms of JS-metric. Hence $(X^2,D_+)$ is a $D_+$-JS-metric space. Proceeding in this way we can define a distance function on any $n$-tuple set $X^n$ for $n\geq 2$.

Next, we define another function $D_m:X^2\rightarrow \mathbb{R^+}$ by $$D_m((x,y),(u,v))=\max\{D(x,u),D(y,v)\}.$$ Then, it can be checked that $D_m$ also satisfies  the axioms of distance function in JS-metric spaces. Hence, $(X^2,D_m)$ is also a $D_m$-JS-metric space. In similar fashion, one can define $n$-tuple $D_m$-JS-metric space for any $n\geq 2.$ In order to state our main result, the following propositions will be necessary.
\begin{proposition}\label{propos1}
 Let $(z_n)=(x_n,y_n)$ be a sequence in $(X^2, D_+)$. Suppose $(z_n)$ $D_+$- converges to $x^*=(x,y)$ and $u^*=(u,v)$. Then $x^*=u^*$.
\end{proposition}
\begin{proof}
\begin{eqnarray}
D_+((x,y),(u,v))&\leq & c\limsup D_+((x_n,y_n),(u,v))\nonumber\\
& \leq & c\limsup \{D(x_n,u)+D(y_n,v)\}=0\nonumber\\
\Rightarrow (x,y)=(u,v).\nonumber
\end{eqnarray}
\end{proof}
\begin{proposition}\label{propos2}
 Let $(x_n)$ be a convergent sequence in $(X,D)$, converges to $x\in X$. Then $D(x,x)=0$.
\end{proposition}
\begin{proof}
By the hypothesis of JS-metric spaces, we can find some $c>0$ such that $$D(x,x)\leq c \limsup_{n\rightarrow \infty} D(x,x_n)=0.$$
\end{proof}
If $(X,D)$ is a complete JS-metric space then one can easily show that $(X^2, D_+)$ and $(X^2, D_m)$ are complete, too.
Let us consider $(x,y)\in X^2$. We define
$$\delta_F(D,(x,y))=\sup\{D(F^i(x,y),F^j(x,y)):i,j\in \mathbb{N}\}$$and
$$\delta_F(D,(y,x))=\sup\{D(F^i(y,x),F^j(y,x)):i,j\in \mathbb{N}\}.$$
Throughout this article, we assume the partial order $`\leq$' on $X^2$ as follows:
$$(u,v)\leq (x,y)\Leftrightarrow u\leq x, v\geq y $$ for all $x,y,u,v\in X$ and we consider $(X^2, D_+)$ as partially ordered complete $D_+$-JS-metric space.
Now, we are in a position to state our main results.
 \subsection{Coupled fixed point theorems in partially ordered JS-metric spaces}
 In this section, we prove the existence and then uniqueness of coupled fixed point for the mappings satisfying mixed monotone property.  
\begin{theorem}\label{thm1}
Let $F:X^2\rightarrow X$ be a mapping with mixed monotone property on $X$. Assume that there exists $k\in [0,1)$ such that $$D(F(x,y),F(u,v))\leq \frac{k}{2}D_+((x,y),(u,v))$$ for $x\geq u;y\leq v$.   If there exist $x_0,y_0\in X$ such that
\begin{enumerate}
\item[(A)] $x_0\leq F(x_0,y_0);y_0\geq F(y_0,x_0)$;
\item[(B)] $\delta_F(D,(x_0,y_0))<\infty$ and $\delta_F(D,(y_0,x_0))<\infty$
\end{enumerate}
then there  exist $x,y\in X$ such that $x=F(x,y);y=F(y,x)$.
\end{theorem}
\begin{proof}
Since $F$ has mixed monotone property so we have
$$x_0\leq F(x_0,y_0)~ and~y_0\geq F(y_0,x_0).$$
Let us consider $x_1=F(x_0,y_0)$ and $y_1=F(y_0,x_0)$ and we also denote
$$F^2(x_0,y_0)=F(F(x_0,y_0),F(y_0,x_0))=F(x_1,y_1)=x_2;$$
$$F^2(y_0,x_0)=F(F(y_0,x_0),F(x_0,y_0))=F(y_1,x_1)=y_2.$$
Processing in this way, by the mixed monotone property of $F$ we can get
$$F^n(x_0,y_0)=F(F^{n-1}(x_0,y_0),F^{n-1}(y_0,x_0))=x_n;$$
$$F^n(y_0,x_0)=F(F^{n-1}(y_0,x_0),F^{n-1}(x_0,y_0))=y_n.$$
It is easy to check that $(x_n)=(F^n(x_0,y_0))$ is monotone increasing sequence and $(y_n)=(F^n(y_0,x_0))$ is monotone decreasing sequence. We claim that both the sequences are Cauchy.

  Let us consider $i,j\in \mathbb{N}$ with $i\leq j$. Then clearly $(x_i,y_i)\leq (x_j,y_j)$.
\begin{eqnarray}
D(x_{1+j},x_{1+i})& = &D(F^{1+j}(x_0,y_0),F^{1+i}(y_0,x_0))\nonumber\\
& =& D(F(F^{j}(x_0,y_0),F^{j}(y_0,x_0)),F(F^{i}(x_0,y_0),F^{i}(y_0,x_0)))\nonumber\\
& \leq & \frac{k}{2}\{D(F^{j}(x_0,y_0),F^{i}(x_0,y_0))+D(F^{j}(y_0,x_0),F^{i}(y_0,x_0))\}\nonumber\\
& \leq & \frac{k}{2}\{D(x_j,x_i)+D(y_j,y_i)\}. \label{inq1}
\end{eqnarray}
Again,
\begin{eqnarray} D(y_{1+j},y_{1+i})& = &D(F^{1+j}(y_0,x_0),F^{1+i}(y_0,x_0))\nonumber\\
& =& D(F(F^{j}(y_0,x_0),F^{j}(x_0,y_0)),F(F^{i}(y_0,x_0),F^{i}(x_0,y_0)))\nonumber\\
& \leq & \frac{k}{2}\{D(F^{j}(y_0,x_0),F^{i}(y_0,x_0))+D(F^{j}(x_0,y_0),F^{i}(x_0,y_0))\}\nonumber\\
& \leq & \frac{k}{2}\{(D(y_j,y_i))+(D(x_j,x_i))\}. \label{inq2}
\end{eqnarray}
Similarly, using the inequalities \ref{inq1} and \ref{inq2}, we obtain
\begin{eqnarray}
D(x_{2+j},x_{2+i})& = &D(F^{2+j}(x_0,y_0),F^{2+i}(y_0,x_0))\nonumber\\
& =& D(F(F^{1+j}(x_0,y_0),F^{1+j}(y_0,x_0)),F(F^{1+i}(x_0,y_0),F^{1+i}(y_0,x_0)))\nonumber\\
& \leq & \frac{k}{2}\{D(x_{1+j},x_{1+i})+D(y_{1+j},y_{1+i})\}\nonumber\\
& \leq & \frac{k^2}{2}\{D(x_j,x_i)+D(y_j,y_i)\}\nonumber
\end{eqnarray}
and
\begin{eqnarray} D(y_{2+j},y_{2+i})& = &D(F^{2+j}(y_0,x_0),F^{2+i}(y_0,x_0))\nonumber\\
& =& D(F(F^{1+j}(y_0,x_0),F^{1+j}(x_0,y_0)),F(F^{1+i}(y_0,x_0),F^{1+i}(x_0,y_0)))\nonumber\\
& \leq & \frac{k}{2}\{D(y_{1+j},y_{1+i})+D(x_{1+j},x_{1+i})\}\nonumber\\
& \leq & \frac{k^2}{2}\{D(x_j,x_i)+D(y_j,y_i)\}.\nonumber
\end{eqnarray}
Now, we assume that $$D(x_{m+j},x_{m+i}) \leq  \frac{k^m}{2}\{D(x_j,x_i)+D(y_j,y_i)\};$$
$$D(y_{m+j},y_{m+i}) \leq  \frac{k^m}{2}\{D(x_j,x_i)+D(y_j,y_i)\}$$ are hold. Then we obtain
\begin{eqnarray}
D(x_{m+1+j},x_{m+1+i})& = &D(F^{m+1+j}(x_0,y_0),F^{m+1+i}(y_0,x_0))\nonumber\\
& =& D(F(F^{m+j}(x_0,y_0),F^{m+j}(y_0,x_0)),F(F^{m+i}(x_0,y_0),F^{m+i}(y_0,x_0)))\nonumber\\
& \leq & \frac{k}{2}\{D(F^{m+j}(x_0,y_0),F^{m+i}(x_0,y_0))+D(F^{m+j}(y_0,x_0),F^{m+i}(y_0,x_0))\}\nonumber\\
& \leq & \frac{k}{2}\{D(x_{m+j},x_{m+i})+D(y_{m+j},y_{m+i})\}\nonumber\\
& \leq & \frac{k^{m+1}}{2}\{D(x_j,x_i)+D(y_j,y_i)\}
\end{eqnarray}
and
\begin{eqnarray} D(y_{m+1+j},y_{m+1+i})& = &D(F^{m+1+j}(y_0,x_0),F^{m+1+i}(y_0,x_0))\nonumber\\
& =& D(F(F^{m+j}(y_0,x_0),F^{m+j}(x_0,y_0)),F(F^{m+i}(y_0,x_0),F^{m+i}(x_0,y_0)))\nonumber\\
& \leq & \frac{k}{2}\{D(F^{m+j}(y_0,x_0),F^{m+i}(y_0,x_0))+D(F^{m+j}(x_0,y_0),F^{m+i}(x_0,y_0))\}\nonumber\\
& \leq & \frac{k}{2}\{D(y_{m+j},y_{m+i})+D(x_{m+j},x_{m+i})\}\nonumber\\
& \leq & \frac{k^{m+1}}{2}\{D(x_j,x_i)+D(y_j,y_i)\}.\nonumber
\end{eqnarray}
Therefore, we can claim that for all $n\in \mathbb{N}$, $$D(x_{n+j},x_{n+i})\leq \frac{k^{n}}{2}\{D(x_j,x_i)+D(y_j,y_i)\};$$ and $$D(y_{n+j},y_{n+i}) \leq \frac{k^{n}}{2}\{D(x_j,x_i)+D(y_j,y_i)\}.$$
By taking limit supremum of both sides of the above inequalities, we get
$$\limsup_{n \rightarrow \infty}D(x_{n+j},x_{n+i})\leq \limsup_{n \rightarrow \infty}\frac{k^{n}}{2}\{\delta_F(D,(x_0,y_0))+\delta_F(D,(y_0,x_0))\};$$
$$\limsup_{n \rightarrow \infty}D(y_{n+j},y_{n+i})\leq \limsup_{n \rightarrow \infty}\frac{k^{n}}{2}\{\delta_F(D,(x_0,y_0))+\delta_F(D,(y_0,x_0))\}.$$ Let us consider $$M=\max \{\delta_F(D,(x_0,y_0)),\delta_F(D,(y_0,x_0))\}.$$ Then from the above inequalities, it is clear that,
\begin{equation}
\limsup_{n \rightarrow \infty}D(x_{n+j},x_{n+i})\leq \limsup_{n \rightarrow \infty} k^nM.
\end{equation}
Since we choose arbitrary values of $i,j$ with $j>i$, therefore for all $p\in \mathbb{N}$, we obtain
$$\limsup_{n \rightarrow \infty}D(x_{n+p},x_{n})\leq \limsup_{n \rightarrow \infty} k^nM.$$
This implies that $(x_n)$ is a Cauchy sequence as $k\in [0,1)$ and $M$ is bounded. Analogously, we can show that $(y_n)$ is also a Cauchy sequence. As $(X^2,D_+)$ is complete so  let $(x_n,y_n)\rightarrow (x,y)\in X^2$ as $n\rightarrow \infty$.

We next claim that $(x,y)$ is coupled fixed point of $F$. Now,
\begin{eqnarray}
D(F(x,y),x) & \leq & c\limsup D(F(x,y),x_n)\nonumber\\
&\leq &c\limsup D(F(x,y),F^n(x_0,y_0))\nonumber\\
&\leq &c\limsup D(F(x,y),F(F^{n-1}(x_0,y_0),F^{n-1}(y_0,x_0))\nonumber\\
&\leq &\frac{kc}{2}\limsup \{D(x,F^{n-1}(x_0,y_0))+D(y,F^{n-1}(y_0,x_0))\}\nonumber\\
&=&0\nonumber\\
\Rightarrow F(x,y)=x.\nonumber
\end{eqnarray}
Analogously, $F(y,x)=y$.
\end{proof}
Our next results show the uniqueness of coupled fixed point.
\begin{theorem}
If $(x,y)$ and $(x^*,y^*)$ are two coupled fixed points of $F$ which are comparable and $D(x,x^*)<\infty$   and $D(y,y^*)<\infty $ then $x=x^*;y=y^*.$
\end{theorem}
\begin{proof}
Since $(x,y)$ and $(x^*,y^*)$ are coupled fixed point so for every $n$, $F^n(x,y)=x;F^n(y,x)=y$ and  $F^n(x^*,y^*)=x^*;F^n(y^*,x^*)=y^*.$ Now,
\begin{eqnarray}
D_+((x,y),(x^*,y^*)) &= & \{D(x,x^*)+D(y,y^*)\}\nonumber\\
&= & \{D(F^n(x,y),F^n(x^*,y^*))+D(F^n(y,x),F^n(y^*,x^*)\}\nonumber\\
&= & \{D(F(F^{n-1}(x,y),F^{n-1}(y,x)),F(F^{n-1}(x^*,y^*),F^{n-1}(y^*,x^*)))\nonumber\\
& &+D(F(F^{n-1}(y,x),F^{n-1}(x,y)),F(F^{n-1}(y^*,x^*),F^{n-1}(x^*,y^*)))\}\nonumber\\
& \leq &\frac{k}{2} [\{D(F^{n-1}(x,y),F^{n-1}(x^*,y^*))+D(F^{n-1}(y,x),F^{n-1}(y^*,x^*))\}\nonumber\\
& &+\{ D(F^{n-1}(y,x),F^{n-1}(y^*,x^*))+D(F^{n-1}(x,y),F^{n-1}(x^*,y^*))\}]\nonumber\\
& \leq & k\{D(F^{n-1}(x,y),F^{n-1}(x^*,y^*))+D(F^{n-1}(y,x),F^{n-1}(y^*,x^*))\}.\nonumber
\end{eqnarray}
Again,
\begin{eqnarray}
D(F^{n-1}(x,y),F^{n-1}(x^*,y^*))& =& D(F(F^{n-2}(x,y),F^{n-2}(y,x)),F(F^{n-2}(x^*,y^*),F^{n-2}(y^*,x^*))\nonumber\\
& \leq & \frac{k}{2}\{D(F^{n-2}(x,y),F^{n-2}(x^*,y^*))+D(F^{n-2}(y,x),F^{n-2}(y^*,x^*))\},\nonumber
\end{eqnarray}
and
\begin{eqnarray}
D(F^{n-1}(y,x),F^{n-1}(y^*,x^*)) &=& D(F(F^{n-2}(y,x),F^{n-2}(x,y)),F(F^{n-2}(y^*,x^*),F^{n-2}(x^*,y^*))\}\nonumber\\
& \leq & \frac{k}{2} \{ D(F^{n-2}(y,x),F^{n-2}(y^*,x^*))+D(F^{n-2}(x,y),F^{n-2}(x^*,y^*))\}.\nonumber
\end{eqnarray}
From the above three inequalities, we get,$$D_+((x,y),(x^*,y^*))\leq k^2[D(F^{n-2}(x,y),F^{n-2}(x^*,y^*))+D(F^{n-2}(y,x),F^{n-2}(y^*,x^*))].$$ Processing in similar way, we obtain
$$D_+((x,y),(x^*,y^*))\leq k^{n-1}[D(x,x^*)+D(y,y^*)]. $$ Since $D(x,x^*)<\infty$ and $D(y,y^*)<\infty$ so $n\rightarrow \infty$ implies that
\begin{eqnarray}
D_+((x,y),(x^*,y^*)) = 0\nonumber\\
\Rightarrow  (x,y)=(x^*,y^*)\nonumber\\
 \Rightarrow  x=x^*; y=y^*.\nonumber
\end{eqnarray}
\end{proof}
Therefore, we can conclude that if there are two comparable couple fixed points with finite distance then indeed they are same. But this does not give guarantee of uniqueness of coupled fixed point since their may be two incomparable coupled fixed point. In this regard, we present the following theorem.
\begin{theorem}
Suppose $(x,y)$ and $(x^*,y^*)$ are two incomparable coupled fixed points of $F$. Let there exist an upper or lower bound $(z_1,z_2)$ of $(x,y)$ and $(x^*,y^*)$. If $D(x,z_1),D(y,z_2), D(x^*,z_1),D(y^*,z_2) <\infty$ then $x=x^*; y=y^*$.
\begin{proof}
Since $(z_1,z_2)$ is upper or lower bound so both $(x,y)$ and $(x^*,y^*)$ are comparable to $(z_1,z_2)$. Hence, for each $n\in \mathbb{N}$, $(x,y)=(F^n(x,y),F^n(y,x))$ and $(x^*,y^*)=(F^n(x^*,y^*),F^n(y^*,x^*))$ are comparable with $(F^n(z_1,z_2),F^n(z_2,z_1)).$\\ Therefore,
\begin{eqnarray}
D_+{\left(\begin{array}{cc}
\begin{pmatrix} x\\y\end{pmatrix},& \begin{pmatrix} F^n(z_1,z_2\\F^n(z_2,z_1)\end{pmatrix}
\end{array}\right)}&=& D_+{\left(\begin{array}{cc}
\begin{pmatrix} F^n(x,y)\\F^n(y,x)\end{pmatrix},& \begin{pmatrix} F^n(z_1,z_2\\F^n(z_2,z_1)\end{pmatrix}
\end{array}\right)}\nonumber\\
& = & D(F^n(x,y),F^n(z_1,z_2))+D(F^n(y,x),F^n(z_2,z_1))\nonumber\\
& = & D(F(F^{n-1}(x,y),F^{n-1}(y,x)),F(F^{n-1}(z_1,z_2),F^{n-1}(z_2,z_1)))\nonumber\\
&  & + D(F(F^{n-1}(y,x),F^{n-1}(x,y)),F(F^{n-1}(z_2,z_1),F^{n-1}(z_1,z_2)))\nonumber\\
& \leq & k \{D(F^{n-1}(x,y),F^{n-1}(z_1,z_2))+D(F^{n-1}(y,x),F^{n-1}(z_2,z_1)) \}\nonumber\\
& \vdots{}& \nonumber\\
& \leq & k^n \{D(x,z_1)+D(y,z_2)\}.\nonumber
\end{eqnarray}
This implies that $(F^n(z_1,z_2),F^n(z_2,z_1))\rightarrow (x,y)$ as $n\rightarrow \infty$, since $D(x,z_1),D(y,z_2)<\infty.$ Similarly, we can see that
\begin{eqnarray}
D_+{\left(\begin{array}{cc}
\begin{pmatrix} x^*\\y^*\end{pmatrix},& \begin{pmatrix} F^n(z_1,z_2\\F^n(z_2,z_1)\end{pmatrix}
\end{array}\right)}&=& D_+{\left(\begin{array}{cc}
\begin{pmatrix} F^n(x^*,y^*)\\F^n(y^*,x^*)\end{pmatrix},& \begin{pmatrix} F^n(z_1,z_2\\F^n(z_2,z_1)\end{pmatrix}
\end{array}\right)}\nonumber\\
& \leq & k^n \{D(x^*,z_1)+D(y^*,z_2)\},\nonumber
\end{eqnarray}
which also shows that $(F^n(z_1,z_2),F^n(z_2,z_1))\rightarrow (x^*,y^*)$ as $n\rightarrow \infty.$
In view of Proposition \ref{propos1}, we must have $(x,y)=(x^*,y^*)$.
\end{proof}
\end{theorem}
\begin{theorem}
In addition to the hypothesis of Theorem \ref{thm1}, we can obtain equality of the component of coupled fixed point i.e. $x=y.$
\begin{proof}
We prove the theorem by considering the following three possible cases.
\textbf{Case I}: Let $x$ be comparable to $y$ with $D(x,y)<\infty$. Then,
$$D(x,y)=D(F(x,y),F(y,x))\leq kD(x,y),$$ This implies that $D(x,y)=0$ i.e. $x=y$.\\
\textbf{Case II}: Let $x$ be not comparable to $y$. Suppose every pair of element in $X$ has an upper or lower bound in $X$. Let $z$ be the upper bound (similarly, one can take it as lower bound) of $x$ and $y$  with $D(z,x)<\infty$ and $D(z,y)<\infty$. Then,  $x\leq z$  and $y\leq z$. Due to mixed monotonicity of $F$, we have
$$F(x,y)\leq F(z,y);F(y,x)\leq F(z,x)~and~F(x,y)\geq F(x,z);F(y,x)\geq F(y,z)$$ Using this with mixed monotone property of $F$, one can yield for all $n\geq 2$,
$$F^n(x,y)\leq F^n(z,y);F^n(y,x)\leq F^n(z,x)$$ and $$F^n(x,y)\geq F^n(x,z);F^n(y,x)\geq F^n(y,z).$$
Now, \begin{equation}
D(x, F^n(x,z))=D(F^n(x,y),F^n(x,z)). \label{x1}
\end{equation}
By using the Proposition \ref{propos2}, we have
\begin{equation}
D(F(x,y),F(x,z))\leq \frac{k}{2}[D(x,x)+D(y,z)]\leq \frac{k}{2}D(y,z)\nonumber
\end{equation}
and
\begin{eqnarray}
D(F^2(x,y),F^2(x,z))&=&D(F(F(x,y),F(y,x)),F(F(x,z),F(z,x)))\nonumber\\
&\leq &\frac{k}{2}[D(F(x,y),F(x,z))+D(F(y,x),F(z,x))]\nonumber\\
&\leq &\frac{k^2}{2}[D(x,x)+D(y,z))]\nonumber\\
&\leq &\frac{k^2}{2}D(y,z).\nonumber
\end{eqnarray} 
Similarly, for each $n>2$, 
\begin{equation}
D(F^n(x,y),F^n(x,z))\leq \frac{k^n}{2}D(y,z).\label{x2}
\end{equation}
In view of Equations \ref{x1} and \ref{x2}, we obtain
\begin{equation}
D(x, F^n(x,z))=D(F^n(x,y),F^n(x,z))\leq \frac{k^n}{2}D(y,z).\label{x3}
\end{equation}
Again, \begin{equation}
D(F^n(z,x),y)= D(F^n(z,x),F^n(y,x)). \nonumber
\end{equation}
\begin{equation}
D(F(z,x),F(y,x))\leq \frac{k}{2}[D(z,y)+D(x,x)]\leq \frac{k}{2}D(y,z).\nonumber
\end{equation}
and \begin{eqnarray}
D(F^2(z,x),F^2(y,x))&=&D(F(F(z,x),F(x,z)),F(F(y,x),F(x,y)))\nonumber\\
&\leq &\frac{k}{2}[D(F(z,x),F(y,x))+D(F(x,z),F(x,y))]\nonumber\\
&\leq &\frac{k^2}{2}[D(z,y)+D(x,x))]\nonumber\\
&\leq &\frac{k^2}{2}D(y,z).\nonumber
\end{eqnarray} 
So, for each $n\geq 2$,
 \begin{equation}
D(F^n(z,x),F^n(y,x))\leq \frac{k^n}{2}D(y,z).\nonumber
\end{equation}
This implies that \begin{equation}
D(F^n(z,x),y)= D(F^n(z,x),F^n(y,x))\leq \frac{k^n}{2}D(y,z).\label{x4}
\end{equation}
Furthermore, $$D(F(x,z),F(z,x))\leq \frac{k}{2}[D(x,z)+D(z,x)]\leq k D(z,x).$$
\begin{eqnarray}
D(F^2(x,z),F^2(z,x))&= & D(F(F(x,z),F(z,x)), F(F(z,x),F(x,z)))\nonumber\\
&\leq &\frac{k}{2}[D(F(x,z),F(z,x))+D(F(z,x),F(x,z))]\nonumber\\
& \leq & k^2 D(x,z).\nonumber
\end{eqnarray}
Processing in this way, we get, 
\begin{equation}
D(F^n(x,z),F^n(z,x))\leq k^n D(z,x)\label{x5}
\end{equation} 
As $D(z,x)<\infty$ and $D(z,y)<\infty$, so from Equations \ref{x3}, \ref{x4} and \ref{x5}, we have
 \[D(x, F^n(x,z))=0, D(F^n(z,x),y)=0 \mbox{ and } D(F^n(x,z),F^n(z,x))=0\] whenever $n\rightarrow \infty$. These imply $x=y$.\\
 
\textbf{Case III}: Let $x_0$ be comparable to $y_0$ with $D(x_0,y_0)<\infty$. Then
\begin{eqnarray}
D(x,y)& \leq & c\limsup D(F^n(x_0,y_0),F^n(y_0,x_0))\nonumber\\
& \leq & c\limsup D(F(F^{n-1}(x_0,y_0),F^{n-1}(y_0,x_0)),F(F^{n-1}(y_0,x_0),F^{n-1}(x_0,y_0)))\nonumber\\
& \leq & kc \limsup {D(F^{n-1}(x_0,y_0),F^{n-1}(y_0,x_0))}\nonumber\\
 & \vdots & \nonumber\\
& \leq & k^{n-1} D(x_0,y_0)\rightarrow 0 ~as~ n\rightarrow \infty.\nonumber
\end{eqnarray}
Therefore, we get $x=y$.
\end{proof}
\end{theorem}
As every standard metric space is JS-metric space so we can obtain the main result of Bhaskar and Lakshmikantham \cite{bhas} as a corollary of our obtained result (Theorem \ref{thm1}).
\begin{corollary}
Let $(X,\leq)$ be a partially ordered set and suppose there is a standard metric $d$ on $X$ such that $(X,d)$ is complete. Let $F:X^2\rightarrow X$ be a mapping with mixed monotone property on $X$.
Assume that there exists $k\in [0,1)$ with
$$d(F(x,y),F(u,v))\leq \frac{k}{2}[d(x,u)+d(y,v)]$$ for each $x\geq u$ and $y\leq v$. If there exist $x_0,y_0\in X$ such that
\begin{enumerate}
\item[(A)]$x_0\leq F(x_0,y_0);y_0\geq F(y_0,x_0)$;
\item[(B)] $\delta_F(D,(x_0,y_0))<\infty$ and $\delta_F(D,(y_0,x_0))<\infty;$
\end{enumerate}
then there  exist $x,y\in X$ such that $x=F(x,y);y=F(y,x)$.
\end{corollary}
\begin{remark}
The authors of \cite{bhas} considered two alternative hypothesises to establish the existence of couple fixed points. These are: either the function $F$ is continuous or if $(x_n)$ and $(y_n)$ are non-decreasing and non-increasing sequences respectively with $x_n\rightarrow x$ and $y_n\rightarrow y$ then $x_n\leq x$; $y_n\geq y$ for all $n\in \mathbb{N}$. But we prove the existence of couple fixed points without assuming any of the above mentioned hypothesises.
\end{remark}
If we replace the distance function $`D_+$' on $X^2$ by $`D_m$' then we can also prove the existence of  coupled fixed point. In this direction we present the following theorem.
\begin{theorem}
Let $F:X^2:\rightarrow X$ be a mapping with mixed monotone property on $X$. Assume that there exists $k\in [0,1)$ such that $$D(F(x,y),F(u,v))\leq kD_m((x,y),(u,v))$$ for $x\geq u;y\leq v$.   If there exist $x_0,y_0\in X$ such that
\begin{enumerate}
\item[(A)] $x_0\leq F(x_0,y_0);y_0\geq F(y_0,x_0)$;
\item[(B)] $\delta_F(D,(x_0,y_0))<\infty$ and $\delta_F(D,(y_0,x_0))<\infty;$
\end{enumerate}
then there  exist $x,y\in X$ such that $x=F(x,y);y=F(y,x)$.
\end{theorem}
\begin{proof} Proof is almost similar to the proof of Theorem \ref{thm1}. Hence, we skip this.
\end{proof}
We now construct an example to support our main result.
\begin{example} Let us consider $X=\mathbb{R}\cup\{\infty, -\infty\}$ and we define the distance function $D$ on $X$ as $D(x,y)=|x|+|y|$ for all $x,y\in X$. At first, we prove that $(X,D)$ is JS-metric space. In order to prove this we check the axioms of JS-metric spaces.
\begin{enumerate}
\item $D(x,y)=0\Rightarrow |x|+|y|=0\Rightarrow |x|=|y|=0$ i.e. $x=y=0.$
\item Clearly, $D(x,y)=D(y,x)$.
\item Let $(x_n)$ be a sequence converging to some $x$ in $X$. Then for any $y\in X$ we have $D(x,y)=|x|+|y|.$ Again, $D(x_n,y)=|x_n|+|y|$ and $\limsup D(x_n,y)=\limsup(|x_n|+|y|)=|x|+|y|$. So, we can always find some $c\geq 1$ such that $D(x,y)\leq c\limsup D(x_n,y)$. 
\end{enumerate}
All the axioms are satisfied. Hence, $(X,D)$ is a  JS-metric space. Now, we consider the metric space $(X^2, D_+)$, where, $$D_+((x,y),(u,v))=D(x,u)+D(y,v).$$ Let us define a function $F:X^2\rightarrow X$ by$$F(x,y)=\frac{x-y}{3} ~\forall x,y\in X.$$ Then,
\begin{enumerate}
\item[(i)] Let $x_1\leq x_2$. Then for all $y\in X$, we have $x_1-y\leq x_2-y$ which implies that $F(x_1,y)\leq F(x_2,y)$ i.e. $F$ is monotonic non-decreasing sequence in its 1st component. Again, for all $x\in X$, whenever $y_1\leq y_2$ we get $x-y_1\geq x-y_2$ which shows that $F(x,y_1)\geq F(x,y_2)$. So $F$ is monotonic non-increasing function in its 2nd component. Thus $F$ has mixed monotone property.
\item[(ii)]
Let $(x,y),(u,v)\in X^2$. Then.
\begin{eqnarray}
D(F(x,y),F(u,v))&= &\frac{|x-y|}{3}+\frac{|u-v|}{3}\nonumber\\
&\leq & \frac{1}{3}|x|+ \frac{1}{3}|y|+  \frac{1}{3}|u|+  \frac{1}{3}|v|\nonumber\\
&\leq &  \frac{1}{3}(|x|+|u|)+ \frac{1}{3}(|y|+|v|)\nonumber\\
& \leq &  \frac{2}{3}\frac{D_+((x,y),(u,v))}{2}\nonumber.
\end{eqnarray}
This shows that $F$ satisfies the contraction condition.
\item[(iii)]
Let us set $x_0=-3$ and $y_0=2$. Then, $$x_1=F(x_0,y_0)= F(-3,2)=\frac{-5}{3}>x_0=-3$$ and 
$$y_1=F(y_0,x_0)= F(2,-3)=\frac{5}{3}<y_0=2.$$ 
Again, it is easy to show that for all $i,j\in \mathbb{N}$, $\delta_F(D,(x_0,y_0))<\infty$ and $\delta_F(D,(y_0,x_0))<\infty$.
\end{enumerate}
 Thus all the conditions of the Theorem \ref{thm1} are satisfied. Therefore $F$ has a coupled fixed point. Here, $(0,0)$ is a coupled fixed point of $F$. Notice that this is not unique since  $(\infty, -\infty)$ is also a coupled fixed point of $F$. 
\end{example}
\subsection{Extension of Berinde's results} In this section we extend the results of Berinde \cite{berd} which generalize the results of Bhaskar and Lakshmikantham \cite{bhas}. The Contraction \ref{c2} in the setting of $(X^2,D_+)$ is presented by 
$$D(F(x,y),F(u,v))+D(F(y,x),F(v,u))\leq k[D(x,u)+D(y,v)]$$ for all $x\geq u;y\leq v$ and $k\in [0,1)$. We define an operator $T_F:X^2\rightarrow X^2$ by 
$$T_F(x,y)=(F(x,y),F(y,x)).$$ Then we can write the above contraction as follows:
\begin{equation}D_+(T_F(X), T_F(U))\leq k D_+(X,U) \label{equ2.1} \end{equation} where $X=(x,y); U=(u,v)$ and $k\in [0,1)$. Therefore, coupled fixed point theorem for $F$ reduces to usual Banach's fixed point theorem for the operator $T_F$ because one can easily check that $F$ has coupled fixed point iff $T_F$ has a fixed point. So, for the existence of coupled fixed point of $F$ it is sufficient to prove that $T_F$ has fixed point in $X^2$.

By the notation $\delta(D_+, T_F,z_0)$, we define $$\delta(D_+, T_F,z_0)=\sup\{ D_+(T_F^i(z_0),T_F^j(z_0));i,j\in \mathbb{N}\}.$$
The following results extend the results of Berinde \cite{berd}.
 \begin{theorem}\label{thm2}
 Let $F:X^2:\rightarrow X$ be a mapping with mixed monotone property on partially ordered complete $D_+$-JS-metric space $(X^2,D_+)$. Suppose for all $x\geq u; y\leq v$, $T_F$ satisfies the Contraction \ref{equ2.1}. Then if there exists $z_0=(x_0,y_0)\in X^2$ such that
 \begin{enumerate}
 \item $x_0\leq F(x_0,y_0)$ and $y_0\geq F(y_0,x_0)$ or 
 \item $x_0\geq F(x_0,y_0)$ and $y_0\leq F(y_0,x_0)$,
 \item $\delta(D_+, T_F,z_0)<\infty$
 \end{enumerate}
 then there exist $\tilde{x},\tilde{y}\in X$ such that $\tilde{x}=F(\tilde{x},\tilde{y})$; $\tilde{y}=F(\tilde{y},\tilde{x})$.
\end{theorem}  
\begin{proof}
By the hypothesis of the theorem, let us assume, there exists $z_0=(x_0,y_0)\in X^2$ with $x_0\leq F(x_0,y_0)$ and $y_0\geq F(y_0,x_0)$. We denote  $x_1=F(x_0,y_0)$ and $y_1=F(y_0,x_0)$. Then $z_1=(x_1,y_1)=(F(x_0,y_0),F(y_0,x_0))=T_F(x_0,y_0)=T_F(z_0)$.\\ Again, $z_2=(x_2,y_2)=(F(x_1,y_1),F(y_1,x_1))=T_F(x_1,y_1)=T_F(z_1)=T_F^2(z_0)$. With this notation, we obtain the Picard sequence $(z_n)$ with initial approximation $z_0$, defined by $$z_{n+1}=T_F(z_n)=(F(x_n,y_n),F(y_n,x_n))=(x_{n+1},y_{n+1}) $$ for all $n\geq 0$ and $z_n=(x_n,y_n)$.\\
Due to mixed monotone property of $F$, it is easy to show that for all $n>0$, $z_{n}\leq z_{n+1}$. Next, we prove that $(z_n)$ is a Cauchy sequence. For all $n\geq 0$ and $i\leq j$, we get
\begin{eqnarray}
D_+(T_F^{n+i}(z_0),T_F^{n+j}(z_0))&\leq & D_+(T_F^{n-1+i}(z_0),T_F^{n-1+j}(z_0))\nonumber\\
\Rightarrow \delta(D_+,T_F,T_F^n(z_0))&\leq & k \delta(D_+,T_F,T_F^{n-1}(z_0))\nonumber\\
&\leq &k^2 \delta(D_+,T_F,T_F^{n-2}(z_0))\nonumber\\
&&\vdots\nonumber\\
& \leq &k^n \delta(D_+,T_F,z_0)\nonumber\\
& \rightarrow & 0 ~as~n \rightarrow \infty.\nonumber
\end{eqnarray}
Therefore, for all $n,m\in \mathbb{N}$, we obtain $$D_+(T_F^{n}(z_0),T_F^{n+m}(z_0))\leq \delta(D_+,T_F,T_F^n(z_0)) \leq k^n \delta(D_+,T_F,z_0)=0,$$ which implies that $(z_n)$ is a Cauchy sequence.
Let $z_n$ converges to $\tilde{z}=(\tilde{x},\tilde{y}) \in X$. Again,\begin{eqnarray}
& &D_+(z_{n+1},T_F(\tilde{z}))=D_+(T_F(z_{n}),T_F(\tilde{z}))\leq k D_+(z_{n},\tilde{z})\nonumber\\
& &\Rightarrow D_+(z_{n+1},T_F(\tilde{z})= 0\nonumber\\
& &\Rightarrow z_n\rightarrow T_F(\tilde{z}) ~as~ n\rightarrow \infty.
\end{eqnarray}
Since limit of a convergent sequence in this structure is unique, so we must have $\tilde{z}=T_F(\tilde{z})$ i.e. $\tilde{z}$ is a fixed point of $T_F$. By previous assertion we can conclude that $\tilde{z}=(\tilde{x},\tilde{y})$ is a coupled fixed point of $F$, that is, $\tilde{x}=F(\tilde{x},\tilde{y})$ and  $\tilde{y}=F(\tilde{y},\tilde{x})$.
\end{proof}
Next, we present the additional conditions for uniqueness of coupled fixed point of $F$.
\begin{theorem}
Let $\tilde{z}=(\tilde{x},\tilde{y})$ and $w=(u,v)$ be two comparable coupled fixed points of $F$ with $D_+(w,\tilde{z})<\infty$. Then $w=\tilde{z}$. 
 \end{theorem}
\begin{proof}We have
\begin{eqnarray}
&&D_+(w,\tilde{z})=D_+(T_F(w),T_F(\tilde{z}))\leq k D_+(w,\tilde{z})\nonumber\\
&&\Rightarrow D_+(w,\tilde{z})=0\nonumber\\
&& \Rightarrow w=\tilde{z}\nonumber\\
&&\Rightarrow (\tilde{x},\tilde{y})=(u,v).\nonumber
\end{eqnarray} 
Hence the proof follows.
\end{proof} 
 \begin{theorem}
 Let $w$ and $\tilde{z}$ be two incomparable coupled fixed points of $F$. Suppose there exists an upper bound or lower bound $z^*=(x^*,y^*)\in X^2$ of $w$ and $\tilde{z}$ with $D_+(w,z^*)<\infty$ and $D_+(\tilde{z},z^*)<\infty$. Then $w=\tilde{z}$. \end{theorem}
 \begin{proof}
Clearly, for every $n\in \mathbb{N}$, $T_F^n(z^*)$ is comparable to $w=T_F^n(w)$ as well as to $\tilde{z}=T_F^n(\tilde{z})$. By contraction principle, we obtain $$D_+(T_F(w),T_F(z^*))\leq k D_+(w,z^*),$$ and
$$D_+(T_F^2(w),T_F^2(z^*))\leq k D_+(T_F(w),T_F(z^*))\leq k^2 D_+(w,z^*).$$ Proceeding in this way, one can obtain,
\begin{equation}
D_+(T_F^n(w),T_F^n(z^*))\leq k^n D_+(w,z^*).\label{*}
\end{equation}
By using the axioms of $D_+$-JS-metric spaces and the above inequality, we have
$$D_+(w,T_F^n(z^*))\leq c \limsup D_+(T_F^n(w),T_F^n(z^*))\leq k^nc D_+(w,z^*).$$ Since, $D_+(w,z^*)<\infty$  and $k\in [0,1)$, $D_+(w,T_F^n(z^*))\rightarrow 0$, whenever $n\rightarrow \infty$. This implies that the sequence $T_F^n(z^*)$ converges to $w$. 

Analogously, it can be prove that the sequence $T_F^n(z^*)$ also converges to $\tilde{z}$. In view of Proposition \ref{propos1} , we must have $\tilde{z}=w$, that is, $(\tilde{x},\tilde{y})=(u,v)$.
 \end{proof}

Next, we are interested to find additional conditions for the equality of the components of coupled fixed point. 
\begin{theorem}
Let $(\tilde{x},\tilde{y})$ be  coupled fixed point of $F$. Suppose every pair of elements in $X$ has either an upper bound or a lower bound in $X$. Then $\tilde{x}=\tilde{y}$.
\end{theorem}
\begin{proof} We prove this theorem in two  possible ways.

\noindent{\textbf{Case-I.}} Let $(\tilde{x},\tilde{y})$ be coupled fixed point of $F$ such that $\tilde{x}$ and $\tilde{y}$ are comparable in $X$ with $D(\tilde{x},\tilde{y})<\infty.$ We consider $X=(\tilde{x},\tilde{y})$ and $U=(\tilde{y},\tilde{x})$. Using the contraction principle in Theorem \ref{thm2}, we get,
\begin{eqnarray}
&&D_+(T_F(X), T_F(U))\leq k D_+(X,U)\nonumber\\
&&\Rightarrow D_+((F(\tilde{x},\tilde{y}),F(\tilde{y},\tilde{x})), (F(\tilde{y},\tilde{x}),F(\tilde{x},\tilde{y}))\leq k D_+((\tilde{x},\tilde{y}),(\tilde{y},\tilde{x}))\nonumber\\
&&\Rightarrow D(F(\tilde{x},\tilde{y}),F(\tilde{y},\tilde{x}))+D(F(\tilde{y},\tilde{x}),F(\tilde{x},\tilde{y}))\leq k (D(\tilde{x},\tilde{y})+D(\tilde{y},\tilde{x}))\nonumber\\
&& \Rightarrow D(F(\tilde{x},\tilde{y}),F(\tilde{y},\tilde{x}))\leq k D(\tilde{x},\tilde{y})\nonumber\\
&& \Rightarrow D(\tilde{x},\tilde{y})\leq k D(\tilde{x},\tilde{y})\nonumber\\
&& \Rightarrow D(\tilde{x},\tilde{y})=0~ i.e. ~\tilde{x}=\tilde{y}.\nonumber
\end{eqnarray}
%\textbf{Case-II}.
%Suppose $\tilde{x}$ and $\tilde{y}$ are not comparable. There exists an upper bound $\tilde{z}$ in $X$ (either lower bound) of $\tilde{x}$ and $\tilde{y}$. Then, $\tilde{x}\leq \tilde{z}$ and $\tilde{y}\leq \tilde{z}.$     \\
\textbf{Case-II}.
Let $x_0,y_0$ are comparable with $D(x_0,y_0)<\infty$. Due to mixed monotonicity of $F$, for each $n\geq 1$, $x_n=F(x_{n-1},y_{n-1})$ and $y_n=F(y_{n-1},x_{n-1})$ are also comparable and $x_n\rightarrow \tilde{x}$ and $y_n\rightarrow \tilde{y}$ as $n\Rightarrow \infty$. By the axioms of JS-metric spaces, we obtain
\begin{equation}
D(x,y)\leq c \limsup D(x_n,y_n). \label{**}
\end{equation}
Again, by taking $X=(x_n,y_n)$ and $U=(y_n,x_n)$ in the contraction condition of Theorem \ref{thm2}, for all $n\geq 0$, we get
\begin{eqnarray}
&&D(F(x_n,y_n),F(y_n,x_n)) \leq k D(x_n,y_n)\nonumber\\
&& \Rightarrow D(x_{n+1}, y_{n+1})\leq k  D(x_n,y_n).\label{***}
\end{eqnarray}
Using Inequalities  \ref{**} and \ref{***}, we must have
$$D(\tilde{x},\tilde{y})\leq c \limsup D(x_n,y_n) \leq k^n cD(x_0,y_0)= 0$$ as $n\rightarrow \infty.$ This implies that $D(\tilde{x},\tilde{y})=0$. Hence, we obtain, $\tilde{x}=\tilde{y}$.
\end{proof}
\vskip.5cm\noindent{\bf Acknowledgement}\\
 This work is funded by DST-INSPIRE, New Delhi, India under INSPIRE fellowship scheme (No. DST/INSPIRE FELLOWSHIP/2013/636). The support is gratefully acknowledged.
%\begin{thebibliography}{20}

%\bibliography{bibfile}
%\bibliographystyle{elsarticle-num}
%\bibliographystyle{plain}
\end{document}